\documentclass[reqno]{amsart}

\usepackage{amssymb}                     
\usepackage{amsmath}                     
\usepackage{amsfonts}
\usepackage{amsthm}                      
\usepackage{tabularx}                    

\newcommand{\tor}{\operatorname{Tor}}

\newcommand{\pd}{\operatorname{pd}}

\newcommand{\length}{\operatorname{\ell}}
\newcommand{\spec}{\operatorname{Spec}}
\newcommand{\supp}{\operatorname{Supp}}
\newcommand{\var}{V}
\renewcommand{\H}{\operatorname{H}}

\newcommand{\codim}{\operatorname{codim}}

\newcommand{\vandim}{\operatorname{vdim}}
\newcommand{\f}[1]{{^{\mathit{f}^{#1}}}\!}
\newcommand{\frob}{F}
\newcommand{\grot}{\operatorname{\mathbb{G}\sf{C}}}
\newcommand{\Grot}{\operatorname{\mathbb{G}\sf{P}}}
\newcommand{\cat}{\operatorname{\sf{C}}}
\newcommand{\Cat}{\operatorname{\sf{P}}}

\newcommand{\id}{\operatorname{id}}

\newcommand{\XXX}{\mathfrak{X}}
\newcommand{\YYY}{\mathfrak{Y}}

\newcommand{\ppp}{\mathfrak{p}}
\newcommand{\qqq}{\mathfrak{q}}
\newcommand{\mmm}{\mathfrak{m}}
\newcommand{\NN}{\mathbb{N}}
\newcommand{\ZZ}{\mathbb{Z}}
\newcommand{\QQ}{\mathbb{Q}}

\newcommand{\comp}[1]{{#1}^c}
\renewcommand{\epsilon}{\varepsilon}
\renewcommand{\phi}{\varphi}
\renewcommand{\tilde}{\widetilde}
\renewcommand{\langle}{[}
\renewcommand{\rangle}{]}

\theoremstyle{plain}
\newtheorem{theo}{Theorem}

\newtheorem{prop}[theo]{Proposition}

\theoremstyle{definition}
\newtheorem{defi}[theo]{Definition}

\newtheorem{rema}[theo]{Remark}

\newtheorem*{assu}{Assumption}
\theoremstyle{remark}

\begin{document}
\title{Diagonalizing the Frobenius}
\author{Esben Bistrup Halvorsen}
\address{Esben Bistrup Halvorsen\\ Department of Mathematical
  Sciences\\ University of Copenhagen\\ Universitetsparken 5\\
  2100 K{\o}benhavn \O\\ Denmark.
}
\email{esben@math.ku.dk}
\thanks{The author is partially supported by FNU, the Danish
  Research Council}
\subjclass[2000]{Primary 13A35, 13D22, 13H15, 14F17}
\keywords{Grothendieck space, Frobenius, vanishing, vanishing dimension, intersection multiplicity, Dutta multiplicity}
\begin{abstract}
Over a Noetherian, local ring $R$ of prime characteristic $p$, the
Frobenius functor $\frob_R$ induces a  diagonalizable map on certain quotients
of rational Grothendieck  groups. This leads to an explicit formula
for the Dutta  multiplicity, and it is shown that a weaker version of
 Serre's vanishing conjecture holds if only
 $\chi(\frob_R(X))=p^{\dim R}\chi(X)$ for all bounded complexes $X$ of
 finitely generated, projective modules with finite length homology.
\end{abstract}

\maketitle

\section{Introduction}
For finitely generated modules $M$ and $N$ over a commutative,
Noetherian, local ring $R$ with  $\pd M<\infty$ and
$\length(M\otimes_R N)<\infty$, the \emph{intersection multiplicity} defined by Serre~\cite{serre} is given by
\[
\chi(M,N) = \sum_i (-1)^i\length(\tor^R_i(M,N)).
\]
The \emph{vanishing conjecture}, also formulated by Serre, states that
\[
\chi(M,N)=0\quad\text{whenever}\quad\dim M +\dim N <\dim R.
\]
Serre's original conjecture requires $R$ to be regular, but the
conjecture makes sense in the more general setting presented above. 
Serre proved that the vanishing conjecture holds when $R$ is regular and of equal
characteristic or unramified of mixed
characteristic. Roberts~\cite{robertsC} and, independently, Gillet
and Soul\'e~\cite{gilletsoule} later proved the conjecture in the more
general setting where the requirement that $R$ be regular is weakened
to the requirement that 
$R$ be a complete intersection and both modules have finite projective
dimension. Foxby~\cite{foxbyM} proved that the conjecture generally holds 
when $\dim N\leq 1$. 

However, the vanishing conjecture does not hold in the full
generality presented above. This was shown in the famous
counterexample by Dutta, Hochster and
McLaughlin~\cite{dhm}. Subsequently, other counterexamples have
emerged, such as the one by Miller and Singh~\cite{millersingh}. 

For rings with prime characteristic $p$, a different intersection multiplicity was introduced by
Dutta~\cite{duttamultiplicity}. The \emph{Dutta multiplicity} is given
when $\dim M+\dim N\leq \dim R$ by 
\[
\chi_\infty(M,N) = \lim_{e\to\infty} \frac{1}{p^{e\codim M}} \chi(\frob_R^e(M),N),
\]
where $\frob_R$ denotes the Frobenius functor. The Dutta multiplicity
satisfies the vanishing conjecture and is equal to the usual
intersection multiplicity whenever this satisfies the vanishing conjecture.

This paper studies the interplay between the vanishing
conjecture and the Frobenius functor. The investigations are performed
by studying \emph{Grothendieck spaces} which are tensor products of
$\QQ$ with homomorphic images of Grothendieck groups of
complexes. Proposition~\ref{vanishing} shows that the class of a bounded
complex of finitely generated, projective modules in a Grothendieck
space satisfies the vanishing conjecture if and only if the 
Frobenius functor acts on it by multiplication by a constant.
Following this is Theorem~\ref{maintheo}, which describes
how to decompose such a class of a complex into eigenvectors for the 
Frobenius. This leads in Remark~\ref{duttacalculations} to the
following formula for the Dutta multiplicity:
\[
\chi_\infty(M,N)= 
\begin{pmatrix} 1 & 0 & \cdots &  0 \end{pmatrix}
 \begin{pmatrix} 1 & 1  & \cdots & 1 \\ 
 p^t & p^{t-1} & \cdots & p^{t-u} \\
 \vdots & \vdots & \ddots & \vdots \\
 p^{ut} & p^{u(t-1)} & \cdots &
 p^{u(t-u)}\end{pmatrix}^{\!\!\!-1} \!\!\!
 \begin{pmatrix} \chi(M,N) \\ \chi(\frob_R(M),N) \\ \vdots \\
   \chi(\frob_R^u(M),N) \end{pmatrix} .
\]
Here, $t$ is the co-dimension of $M$ and $u$ is a number that, in a sense, measures how
far $M$ is from satisfying the vanishing conjecture.
The formula can be useful, for example when using a computer to calculate Dutta multiplicity. 
It should be noted that the diagonalizability of
the Frobenius functor has been discussed by Kurano~\cite{KuranoRR}, but that
the approach taken and the results obtained in this paper are new, at least to the knowledge of this
author.

The last section of this paper introduces the concept of \emph{numerical vanishing},
a condition which holds if the  vanishing conjecture holds, and
which implies a weaker version of the vanishing
conjecture, namely the one in which both modules are required to have
finite projective dimension. A consequence of the investigations
performed is the result from Remark~\ref{lengthrema2} that the 
weak vanishing conjecture holds
if only $\chi(\frob_R(X))=p^{\dim R}\chi(X)$ for all bounded
complexes $X$ of finitely generated, projective modules with finite length homology.

\section{Notation}
Throughout this paper, $R$ denotes a commutative, Noetherian,
local ring with maximal ideal $\mmm$ and residue field
$k=R/\mmm$. Modules and complexes are, unless otherwise stated,
assumed to be $R$-modules and $R$-complexes, respectively. Modules are
considered to be complexes concentrated in degree zero.

The \emph{spectrum} of $R$, denoted $\spec R$, is the set of prime ideals of $R$. A subset
$\XXX\subseteq\spec R$ is
\emph{specialization-closed} if, for any inclusion $\ppp\subseteq\qqq$ of
prime ideals, $\ppp\in\XXX$ implies $\qqq\in\XXX$. A closed subset of
$\spec R$ is, in particular, specialization-closed.
 Throughout, whenever
we deal with subsets of the spectrum of a ring, it is implicitly assumed that
they are non-empty and specialization-closed.

For every $\XXX\subseteq\spec R$, the \emph{dimension} of
$\XXX$, denoted $\dim\XXX$, is the usual Krull dimension of
$\XXX$, and the \emph{co-dimension} of $\XXX$, denoted $\codim\XXX$,
is the number $\dim R -\dim\XXX$. The dimension and
co-dimension of a complex $X$ (and hence also of a module) is the
dimension and co-dimension of its support: that is, of the set
$\supp_R X=\{\ppp\in\spec R\mid \H(X_\ppp)\neq 0\}$.

\section{Grothendieck spaces and vanishing}
For every (non-empty, specialization-closed) $\XXX\subseteq\spec R$, consider the following categories:
\begin{center}
\begin{tabularx}{\textwidth}{rcX}
$\Cat(\XXX)$ &=& the category of bounded complexes with support
contained in $\XXX$ and consisting of finitely generated, projective modules.\\
$\cat(\XXX)$ &=& the category of homologically bounded complexes with
support contained in $\XXX$ and with finitely generated homology modules.
\end{tabularx}
\end{center}
If $\XXX=\{\mmm\}$, we simply write $\Cat(\mmm)$ and $\cat(\mmm)$.

The \emph{Euler characteristic} of a complex $X$ in $\cat(\mmm)$ is
the integer
\[
\chi(X)= \sum_{i}(-1)^i \length(\H_i(X)) .
\]
If $M$
and $N$ are finitely generated modules with $\pd M<\infty$ and
$\length(M\otimes_R N)<\infty$, and $X$ is a projective resolution of
$M$, $X\otimes_R N$ is a complex in $\cat(\mmm)$, and the intersection
multiplicity $\chi(M,N)$ of $M$ and $N$ is the number
$\chi(X\otimes_R N)$. There is no problem in letting $N$ be a complex rather than just
a module, so the definition of intersection multiplicity can be extended
to an even more general setting: for subsets $\XXX,\YYY\subseteq\spec R$ with
$\XXX\cap\YYY=\{\mmm\}$ and complexes $X\in\Cat(\XXX)$ and
$Y\in\cat(\YYY)$, the intersection multiplicity of $X$ and
$Y$ is defined as
\[
\chi(X,Y)=\chi(X\otimes_R Y)= \sum_{i}(-1)^i \length(\H_i(X\otimes_R
Y)) .
\]
In the construction of Grothendieck spaces below, the extra requirement that $\dim\XXX+\dim\YYY\leq
\dim R$ is needed;  this corresponds to the assumption that $\dim M +\dim N \leq \dim R$, which is necessary in order to define the Dutta
multiplicity. To formalize this, define, for each $\XXX\subseteq\spec R$, the subset
\[
\comp{\XXX}=\big\{\qqq\in\spec R\mid \XXX\cap\var(\qqq)=\{\mmm\}\text{ and } \dim
\var(\qqq)\leq\codim\XXX\big\}.
\]
The set $\comp\XXX$ is the largest  specialization-closed subset of $\spec R$ such that
\[
\XXX\cap\comp{\XXX}=\{\mmm\}\quad \text{and}\quad
\dim\XXX+\dim\comp{\XXX}\leq\dim R.
\]
(It is not hard to see that, when $\XXX$ is closed, $\dim \XXX+\dim\comp\XXX=\dim
R$.) Thus, for $\XXX,\YYY\subseteq\spec R$, the property that
$\XXX\cap\YYY=\{\mmm\}$ and $\dim\XXX+\dim\YYY\leq\dim R$ is 
equivalent to $\YYY\subseteq\comp\XXX$ which again is equivalent to $\XXX\subseteq\comp\YYY$.

\begin{defi}
Let $\XXX\subseteq\spec R$.
The \emph{Grothendieck space} of the category $\Cat(\XXX)$ is the
$\QQ$-vector space  $\Grot(\XXX)$ presented by elements $[X]$,
one for each isomorphism class of a complex $X$ in $\Cat(\XXX)$, and relations
\[
[X]=[\tilde X] \quad\text{whenever}\quad
\chi(X,-)=\chi(\tilde X, -)\colon \cat(\comp{\XXX})\to\QQ.
\]
Similarly, the \emph{Grothendieck space} of the category $\cat(\XXX)$
is the $\QQ$-vector space $\grot(\XXX)$ presented by elements $\langle
Y\rangle$, one for each isomorphism class of a complex $Y$ in $\cat(\XXX)$, and
relations 
\[
\langle Y\rangle = \langle \tilde Y\rangle\quad\text{whenever}\quad
\chi(-,Y)=\chi(-,\tilde Y)\colon \Cat(\comp{\XXX})\to\QQ.
\]
If $\XXX=\{\mmm\}$,
we simply write $\Grot(\mmm)$ and $\grot(\mmm)$.
\end{defi}

Since intersection multiplicity is additive on short exact sequences and
trivial on exact complexes, the Grothendieck spaces $\Grot(\XXX)$ and
$\grot(\XXX)$ can also be regarded as the tensor product of $\QQ$ with
quotients of the Grothendieck groups $K_0(\Cat(\XXX))$ and
$K_0(\cat(\XXX))$ of the categories $\Cat(\XXX)$ and
$\cat(\XXX)$. (For further details on Grothendieck groups of
categories of complexes, see~\cite{foxbyhalvorsen}.) In
particular, any relation in one of these Grothendieck
groups is also a relation in the corresponding Grothendieck space.

Intersection multiplicity in one variable naturally induces $\QQ$-linear maps
\[
\chi(-,Y) \colon \Grot(\XXX)\to\QQ\quad\text{given by}\quad \chi([X],Y)=\chi(X,Y)
\]
for each $Y\in\cat(\comp\XXX)$. We equip $\Grot(\XXX)$ with the
initial topology of these maps: this is the coarsest topology such
that all the maps are continuous. Likewise, there are naturally
induced $\QQ$-linear maps 
\[
\chi(X,-) \colon \grot(\XXX)\to\QQ\quad\text{given by}\quad \chi(X,[Y])=\chi(X,Y)
\]
for each $X\in\Cat(\comp\XXX)$, and we equip $\grot(\XXX)$ with the
initial topology of these maps. It is straigthforward to see that
addition and scalar multiplication are continuous operations, making
$\Grot(\XXX)$ and $\grot(\XXX)$ topological $\QQ$-vector
spaces. Henceforth, Grothendieck spaces are always considered to be
topological $\QQ$-vector spaces, so that, for example,  a ``homomorphism'' between
Grothendieck spaces is a continuous and $\QQ$-linear map. 

\begin{prop} \label{grothendieckobs}
Suppose that $\XXX,\YYY\subseteq\spec R$.
\begin{enumerate}
\item \label{ses} If $0\to X\to Y\to Z\to 0$ is a short exact sequence of
  complexes in $\Cat(\XXX)$ (or in $\cat(\XXX)$, respectively), then
  $[Y]=[X]+[Z]$ in $\Grot(\XXX)$ (or in $\grot(\XXX)$, respectively).
\item \label{quasi} If $\phi\colon X\to Y$ is a quasi-isomorphism of complexes in
  $\Cat(\XXX)$ (or in $\cat(\XXX)$, respectively), then $[X] =
  [Y]$ in $\Grot(\XXX)$ (or in $\grot(\XXX)$, respectively). In particular, if $X$ is exact, then $[X]=0$.
\item \label{shift} If $X$ is a complex in $\Cat(\XXX)$ (or in $\cat(\XXX)$,
  respectively), then $[\Sigma^n X] = (-1)^n[X]$ in $\Grot(\XXX)$ (or
  in $\grot(\XXX)$, respectively). (Here, $\Sigma^n(-)$ denotes the
  shift functor, taking 
  a complex $X$ to the complex $\Sigma^n X$ defined by $(\Sigma^n X)_i
  = X_{i-n}$ and $\partial^{\Sigma^n X}_i= (-1)^n\partial^X_{i-n}$.)
\item \label{intheform} Any element in $\Grot(\XXX)$ (or in $\grot(\XXX)$, respectively)
  can be written in the form $r[X]$ for a rational number $r\in\QQ$ and a complex $X$ in
  $\Cat(\XXX)$ (or in $\cat(\XXX)$, respectively). 
\item \label{generatedg} $\grot(\XXX)$ is generated by the elements
  $\langle R/\qqq\rangle$ for prime ideals $\qqq\in\XXX$. 
\item \label{iso} The Euler characteristic $\chi\colon\cat(\mmm)\to\QQ$
  induces an isomorphism (that is, a $\QQ$-linear homeomorphism) 
\[
\chi\colon\grot(\mmm)\overset{\cong}{\to}\QQ\quad\text{given by}\quad
\chi([X])=\chi(X).
\]
\item \label{inclusion} The inclusion $\Cat(\XXX)\to\cat(\XXX)$ and,
  when $\XXX\subseteq\YYY$, the inclusions $\Cat(\XXX)\to\Cat(\YYY)$ and $\cat(\XXX)\to\cat(\YYY)$ of
  categories induce homomorphisms $\Grot(\XXX)\to\grot(\XXX)$,
  $\Grot(\XXX)\to\Grot(\YYY)$ and $\grot(\XXX)\to\grot(\YYY)$ given in
  all cases by   $[X]\mapsto [X]$.
\item \label{tensor} If $\YYY\subseteq\comp\XXX$, the tensor product
  of complexes induces bi-homomorphisms (homomorphisms in each variable)
\begin{gather*}
  -\otimes -\colon\Grot(\XXX)\times \grot(\YYY)\to \grot(\mmm)\rlap{\quad\text{and}} \\
  -\otimes -\colon \Grot(\XXX)\times \Grot(\YYY)\to \Grot(\mmm)
\end{gather*}
given in both cases by $[X]\otimes\langle
Y\rangle=\langle X\otimes_R  Y\rangle$.
\end{enumerate}
\end{prop}
\begin{proof}
Properties \eqref{ses}, \eqref{quasi} and \eqref{shift} hold since they hold for the
corresponding Grothendieck groups; see~\cite{foxbyhalvorsen}.

We show that~\eqref{intheform}
holds for elements in $\Grot(\XXX)$; the argument for elements in
$\grot(\XXX)$ is identical. Note first that any element in 
$\Grot(\XXX)$ can be written as a sum $\sum_i r_i[X^i]$ for
various complexes $X^i$ in $\Cat(\XXX)$. By using
\eqref{shift}, we can assume that all $r_i$ are positive, and by
choosing a greatest common divisor, we can write the element in the
form $r\sum_i a_i[X^i]$ for a rational number $r$ and
positive integers $a_i$. Because of~\eqref{ses}, a sum of two elements
represented by complexes is equal to the element represented by their
direct sum, and hence the sum $\sum_i a_i[X^i]$ can be replaced by a
single element $[X]$, where $X$ is the direct sum over $i$ of
$a_i$ copies of $X^i$.

Property~\eqref{generatedg} holds since it holds for the corresponding
Grothendieck group. This is easily seen by using short exact sequences
to transform a complex in $\cat(\XXX)$ first into a bounded complex,
then into the alternating sum of its homology
modules, and finally, by taking filtrations, into a linear combination
of modules in the form $R/\qqq$ for prime ideals $\qqq\in\XXX$.

The $\QQ$-vector space isomorphism in~\eqref{iso} is an immediate
consequence of the group isomorphism $K_0(\cat(\mmm))\overset{\cong}{\to}\ZZ$ induced by the Euler characteristic on
Grothendieck groups. It is straightforward to see that it is a
homeomorphism.

To see~\eqref{inclusion}, it
suffices to note that, since $\cat(\comp{\XXX})$ contains
$\Cat(\comp{\XXX})$ as well as $\cat(\comp{\YYY})$ whenever
$\XXX\subseteq\YYY$ (because then $\comp{\YYY}\subseteq\comp{\XXX}$), any relation in $\Grot(\XXX)$ is also a relation
in $\grot(\XXX)$ and $\Grot(\YYY)$.

Finally, \eqref{tensor} simply follows from the definition of Grothendieck
spaces. As an example, we show that the second map in~\eqref{tensor}
is a homomorphism in the first variable. So fix
$Y\in\Cat(\YYY)$ and let $Z\in\cat(\comp{\{\mmm\}})=\cat(\spec R)$ be
arbitrary. Then
\[
\chi(-\otimes_R Y,Z)=\chi(-,Y\otimes_R Z)\colon \Cat(\XXX)\to\QQ,
\]
which shows that the map $\Grot(\XXX)\to\Grot(\mmm)$ given by $[X]\mapsto [X\otimes_R
Y]$ is well-defined, $\QQ$-linear and continuous.
\end{proof}
The homomorphisms in
Proposition~\ref{grothendieckobs}\eqref{inclusion} are called \emph{inclusion homomorphisms} although they in general are not injective. The image under an inclusion
homomorphism of an element $\alpha$ will generally be denoted $\overline\alpha$.

Let $\XXX,\YYY\subseteq\spec R$ with $\YYY\subseteq\comp\XXX$ and
suppose that $X\in\Cat(\XXX)$ and $Y\in\cat(\YYY)$. Then
\[
\chi(X,Y) = \chi(X\otimes_R Y) = \chi(\langle X\otimes_R Y
\rangle ) = \chi([X]\otimes\langle Y\rangle),
\]
which is the image in $\QQ$ of $[X]\otimes\langle Y\rangle$ under the
isomorphism $\grot(\mmm)\cong \QQ$ induced by the Euler characteristic.
Thus, the intersection multiplicity of complexes generalizes to the bi-homomorphism
$\Grot(\XXX)\times\grot(\YYY)\to \grot(\mmm)$ from Proposition~\ref{grothendieckobs}\eqref{tensor}.

\begin{defi}
Given $\XXX\subseteq\spec R$ and elements $\alpha\in\Grot(\XXX)$ and $\beta\in\grot(\XXX)$, the
\emph{dimensions} of $\alpha$ and $\beta$ are defined as
\begin{gather*}
\dim\alpha=\inf\big\{\dim X\mid \text{$\alpha=r[X]$ for some
  $r\in\QQ$ and $X\in\Cat(\XXX)$}\big\}\rlap{ and}\\
\dim\beta=\inf\big\{\dim Y\mid \text{$\beta=s[Y]$ for some
  $s\in\QQ$ and $Y\in\cat(\XXX)$}\big\}.
\end{gather*}
In particular, $\dim\alpha=-\infty$ if and only if $\alpha=0$.
\end{defi}

\begin{defi}
Suppose that $\XXX\subseteq\spec R$ and let
$\alpha\in\Grot(\XXX)$. Then $\alpha$ \emph{satisfies vanishing} if, for all 
$\beta\in\grot(\comp{\XXX})$,
$\alpha\otimes\beta=0$ whenever $\dim\beta <\codim \XXX$, and $\alpha$
satisfies \emph{weak vanishing} if, for all
$\beta\in\Grot(\comp{\XXX})$, $\overline{\alpha\otimes\beta} =0$  in
$\grot(\mmm)$ whenever $\dim\beta<\codim\XXX$.
The \emph{vanishing dimension} of
  $\alpha$ is the number
\[
\vandim \alpha = \inf\strut \Big\{u\in\ZZ \,\Big\vert\,
\begin{matrix}\text{$\alpha\otimes\beta=0$ for all
    $\beta\in\grot(\comp\XXX)$}\\ \text{with $\dim\beta<\codim\XXX-u$}\end{matrix}\Big\}.
\]
In particular, $\vandim\alpha=-\infty$ if and only if $\alpha=0$, and
$\vandim\alpha\leq 0$ if and only if $\alpha$ satisfies vanishing.
\end{defi}

To satisfy vanishing and weak vanishing for an element $\alpha$
generalizes the usual terminology for complexes: if $X\in\Cat(\XXX)$,
then the element $[X]$ in $\Grot(\XXX)$ 
satisfies vanishing exactly when $\chi(X,Y)=0$ for all 
$Y\in\cat(\comp\XXX)$. Likewise, $[X]$
satisfies weak vanishing exactly when $\chi(X,Y)=0$ for all $Y\in\Cat(\comp\XXX)$.

The vanishing dimension measures, in a sense, 
how far an element is from satisfying vanishing: if $\vandim
[X]=u$, then $u$ is the smallest integer such that
$\chi(X,Y)=0$ for all $Y\in\cat(\comp\XXX)$ with $\dim X+\dim Y<\dim R-u$.

\begin{rema}\label{foxbyexam}
A result by Foxby~\cite{foxbyM} shows that vanishing holds for all
$\alpha\in\Grot(\XXX)$ whenever $\codim\XXX \leq 2$. In particular, for all $\alpha\in\Grot(\XXX)$,
\[
\vandim\alpha\leq \max(0,\codim\XXX-2).
\]
\end{rema}

\begin{prop}\label{conditions}
Suppose that $\XXX\subseteq\spec R$, let
$\alpha\in\Grot(\XXX)$ and let $u$ be a non-negative integer. The
following are equivalent.
\begin{enumerate}
\item \label{D1} $\alpha\otimes\beta=0$ for all $\beta\in\grot(\comp{\XXX})$ with
  $\dim\beta <\codim\XXX-u$.
\item \label{Dx} $\overline\alpha$ satisfies vanishing in $\Grot(\YYY)$ for all
  $\YYY\supseteq\XXX$ with $\codim\YYY =\codim\XXX-u$.
\item \label{D4} $\overline\alpha=0$ in $\Grot(\YYY)$ for all
  $\YYY\supseteq\XXX$ with $\codim\YYY <\codim\XXX-u$.
\item \label{D5} $\overline\alpha=0$ in $\Grot(\YYY)$ for all
  $\YYY\supseteq\XXX$ with $\codim\YYY = \codim\XXX -u-1$.
\item $\vandim\alpha \leq u$.
\end{enumerate}
\end{prop}
\begin{proof} 
Straightforward.
\end{proof}

\begin{rema}\label{vdimineq} Suppose that $\XXX\subseteq\YYY$,
  let $\alpha\in\Grot(\XXX)$ and denote by $\overline\alpha$ the image in
  $\Grot(\YYY)$ of $\alpha$ under the inclusion homomorphism. Then 
\[
\vandim\overline\alpha \leq \vandim\alpha - (
\codim\XXX - \codim\YYY).
\]
It is always possible to find a $\YYY\supseteq\XXX$ with
any given co-dimension larger than or equal to $\codim\XXX-\vandim\alpha$ and
smaller than or equal to $\codim\XXX$ such that the above is
an equality.
\end{rema}

\section{Frobenius and vanishing dimension}
\begin{assu}
Throughout this section, $R$ is assumed to be complete of prime
characteristic $p$, and $k$ is assumed to be a perfect
field.\footnote{Note that, although the assumptions that $R$ be
  complete 
and $k$ be perfect may seem
restrictive, they really are not when it comes to dealing with
intersection multiplicities; for further details, see Dutta~\cite[p.\ 425]{duttamultiplicity}.}
\end{assu}

The Frobenius ring homomorphism $f\colon R\to R$ is given by
$f(r)=r^p$; the $e$-fold composition of $f$ is the ring
homomorphism $f^e\colon R\to R$ given by $f(r)=r^{p^e}$. We denote $\f{e}R$ the bi-$R$-algebra $R$ having the structure of an $R$-algebra
from the left by $f^e$ and from the right by the identity map: that
is, if $x\in \f{e}R$ and $r,s\in R$, then $r\cdot x\cdot s =
r^{p^e}xs$.

\begin{defi}
Two functors, $\f{e}(-)$ and $\frob_R^e$, are defined on the category of $R$-modules by
\[
\f{e}(-) = \f{e}R\otimes_R - \quad \text{and} \quad \frob^e_R(-)=-\otimes_R \f{e}R,
\]
where, for a module $M$, $\f{e}M$ is viewed through its
\emph{left} structure, whereas $\frob_R^e(M)$ is viewed through its
\emph{right} structure.
The functor $\frob_R$ is called the \emph{Frobenius functor}.
\end{defi}

Like the usual intersection multiplicity, the definition of Dutta
multiplicity can be extended to a more general setting: for subsets 
$\XXX,\YYY\subseteq\spec R$ with $\YYY\subseteq\comp\XXX$ and complexes
$X\in\Cat(\XXX)$ and $Y\in\cat(\YYY)$, the Dutta multiplicity
of $X$ and $Y$ is defined as
\[
\chi_\infty(X,Y)=\lim_{e\to\infty}\frac{1}{p^{e\codim X}}\chi(\frob_R^e(X),Y).
\]

\begin{prop}\label{frobeniusobs}
The following hold.
\begin{enumerate}
\item \label{exact}For all $\XXX\subseteq\spec R$, $\f{e}(-)$ defines an exact
  functor $\cat(\XXX)\to\cat(\XXX)$.
\item For all $\XXX\subseteq\spec R$, $\frob_R$ defines a functor
  $\Cat(\XXX)\to\Cat(\XXX)$.
\item $\f{e}(-)$ and $\frob_R^e$ are the compositions of $e$ copies of $\f{}(-)$ and $\frob_R$, respectively.
\end{enumerate}
\end{prop}
\begin{proof}
All properties are readily verified. For further details, see, for
example, Peskine and Szpiro~\cite{peskineszpiro} or Roberts~\cite{roberts}.
\end{proof}

According to Proposition~\ref{frobeniusobs}\eqref{exact}, for any complex $Z\in\cat(\mmm)$,
\[
\chi(\f{e}Z)=\chi(Z)\length(\f{e}k)=\chi(Z),
\]
where the last equation follows since $k\cong \f{e}k$. Now, suppose
that $X\in\Cat(\XXX)$ and $Y\in\cat(\comp{\XXX})$. It is not hard
to see that $\f{e}(\frob_R^e(X)\otimes_R Y)\cong X\otimes_R \f{e}Y$,
and it 
follows that
\begin{equation}\label{noget}
\chi(\frob_R^e(-), Y) = \chi(-, \f{e}Y)\colon\Cat(\XXX)\to\QQ,
\end{equation}
which implies that the map $\Grot(\XXX)\to\Grot(\XXX)$ given by
$[X]\mapsto [\frob^e_R(X)]$ is well-defined, $\QQ$-linear and
continuous; in other words, it is an endomorphism of Grothendieck
spaces. 

\begin{defi}
Given $\XXX\subseteq\spec R$ and $e\in\NN_0$, the endomorphism on $\Grot(\XXX)$ induced by $\frob_R^e$ is
denoted $\frob_\XXX^e$. Further, we define the endomorphism
\[
\Phi^e_\XXX = \frac{1}{p^{e\codim \XXX}}\frob_\XXX^e
\]
on $\Grot(\XXX)$. For $\XXX=\{\mmm\}$ we simply write $\frob_\mmm^e$ and
$\Phi^e_\mmm$.
\end{defi}

\begin{prop}\label{vanishing}
Suppose that $\XXX\subseteq\spec R$ and let $\alpha\in\Grot(\XXX)$. Then
$\alpha$ satisfies vanishing if and only if $\alpha = \Phi_\XXX(\alpha)$.
\end{prop}
\begin{proof}
According to Proposition~\ref{grothendieckobs}\eqref{intheform}, we can assume that
$\alpha$ is in the form $\alpha=r[X]$ for $r\in\QQ$ and  $X\in\Cat(\XXX)$. By
Proposition~\ref{grothendieckobs}\eqref{generatedg} and the definition
of Grothendieck spaces, the element $\alpha$ is completely determined
by the intersection multiplicities $\chi(\alpha,R/\qqq)$ for prime
ideals $\qqq\in\comp\XXX$. Given such a prime ideal $\qqq$, set $m=\dim R/\qqq$
and note that, since $R/\qqq$ is a complete domain of characteristic
$p$ and with
perfect residue field, $R/\qqq$ is torsion-free of rank $p^m$ over
$\f{}(R/\qqq)$; see Roberts~\cite[section~7.3]{roberts}. Thus, there
is a short exact sequence
\[
0\to (R/\qqq)^{p^m}\to \f{}(R/\qqq)\to Q\to 0,
\]
where $Q$ is a finitely generated module with $\dim Q<m$. By applying~\eqref{noget}, we get 
\[
\chi(\frob_R(X),R/\qqq)=p^m\chi(X,R/\qqq)+\chi(X,Q).
\]
Setting $t=\codim\XXX\geq m$, this means that
\begin{equation}\label{phialpha}
\chi(\Phi_\XXX(\alpha),R/\qqq)=p^{m-t}\chi(\alpha,R/\qqq)+p^{-t}\chi(\alpha,Q).
\end{equation}

Now, if $\alpha$ satisfies vanishing, formula~\eqref{phialpha} shows
that $\alpha$ and $\Phi_\XXX(\alpha)$ yield the same intersection
multiplicities with $R/\qqq$ for all $\qqq\in\comp\XXX$, which means that
$\alpha=\Phi_\XXX(\alpha)$. Conversely, if $\alpha=\Phi_\XXX(\alpha)$,
then formula~\eqref{phialpha} implies that
\[
(p^t-p^m)\chi(\alpha,R/\qqq)=\chi(\alpha,Q),
\]
which means that $\alpha$ satisfies vanishing: for if this were not
the case, one could choose $\qqq\in\comp\XXX$ with $m=\dim R/\qqq<t$ minimal such
that $\chi(\alpha,R/\qqq)\neq 0$, and minimality of $m$ would then imply that
$\chi(\alpha,Q)=0$ which gives a contradiction.
\end{proof}

\begin{theo}\label{maintheo}
Suppose that $\XXX\subseteq\spec R$, let
$\alpha\in\Grot(\XXX)$ and suppose that $u$ is a
non-negative integer with $u\geq \vandim\alpha$. Then
\begin{equation}\label{recursionformula}
(p^u\Phi_\XXX-\id)\circ \cdots \circ (p\Phi_\XXX-\id) \circ (\Phi_\XXX-\id)(\alpha) =0.
\end{equation}
Further, there exists a decomposition $\alpha=\alpha^{(0)}+\cdots +
\alpha^{(u)}$ in which
each $\alpha^{(i)}$ is either zero or an eigenvector for $\Phi_\XXX$
with eigenvalue $1/p^i$. The elements
$\alpha^{(0)},\dots ,\alpha^{(u)}$ can be recursively defined by
\[
\alpha^{(0)}=\lim_{e\to\infty} \Phi_\XXX^e(\alpha) \quad\text{and}\quad \alpha^{(i)}=
\lim_{e\to\infty}p^{ie}\Phi^e_\XXX(\alpha-(\alpha^{(0)}+\cdots +\alpha^{(i-1)})),
\]
and there is a formula
\begin{equation}\label{generalformula}
\begin{pmatrix} \alpha^{(0)} \\ \vdots \\ \alpha^{(u)} \end{pmatrix} =
 \begin{pmatrix} 1 & 1  & \cdots & 1 \\ 
 1 & 1/p & \cdots & 1/p^u \\
 \vdots & \vdots & \ddots & \vdots \\
 1 & 1/p^{u} & \cdots &
 1/p^{u^2}\end{pmatrix}^{\!\!\!-1} \!\!\!
 \begin{pmatrix} \alpha \\ \Phi_\XXX(\alpha) \\ \vdots \\
   \Phi_\XXX^u(\alpha) \end{pmatrix} .
\end{equation}
\end{theo}
\begin{proof}
We prove \eqref{recursionformula} by induction on $u$. The case $u =0$ is
trivial since Proposition~\ref{vanishing} in this situation yields that
$(\Phi_\XXX-\id)(\alpha)=0$. Now, suppose that $u>0$ and that the
formula holds for smaller values of $u$. By
Proposition~\ref{vanishing} and commutativity of the involved maps, equation
\eqref{recursionformula} holds if and only if vanishing holds for the element
\[
\beta = (p^u\Phi_\XXX-\id)\circ \cdots \circ (p\Phi_\XXX-\id)(\alpha) .
\]
Now, let $\YYY\subseteq\spec R$ with 
$\YYY\supseteq\XXX$ and $\codim\YYY=\codim\XXX-1$. Then, in
$\Grot(\YYY)$, $\overline{\Phi_\XXX(\alpha)}=p^{-1}\Phi_{\YYY}(\overline\alpha)$,
and hence
\[
\overline\beta = (p^{u-1}\Phi_{\YYY}-\id)\circ \cdots \circ
(p\Phi_{\YYY}-\id) \circ 
(\Phi_{\YYY}-\id)(\overline\alpha) =0 ,
\]
where the last equation follows by induction, since
$\vandim\overline\alpha\leq u-1$ by Remark~\ref{vdimineq}. According
to Proposition~\ref{conditions}, this
proves that $\beta$ satisfies vanishing.

By applying $\Phi_\XXX^{e-u}$ to \eqref{recursionformula}, we get a recursive formula to compute
$\Phi_\XXX^{e+1}(\alpha)$ from $\Phi_\XXX^e(\alpha),\dots
,\Phi_\XXX^{e-u}(\alpha)$. The characteristic polynomial for the
recursion is 
\[
(p^ux-1)\cdots (px-1)(x-1),
\]
which has $u+1$ distinct roots, namely $1,1/p,\dots ,1/p^u$. Thus, there is a general formula
\begin{equation}\label{oldformula}
\Phi_\XXX^e(\alpha) = \alpha^{(0)} + \frac{1}{p^e}\alpha^{(1)} +\cdots +\frac{1}{p^{ue}}\alpha^{(u)}
\end{equation}
for suitable $\alpha^{(0)} ,\dots ,\alpha^{(u)}\in\Grot(\XXX)$, where each
$\alpha^{(i)}$ satisfies
\begin{equation}\label{phiongamma}
\Phi^e_\XXX(\alpha^{(i)})=\frac{1}{p^{ei}}\alpha^{(i)}
\end{equation}
and hence is an eigenvector for $\Phi_\XXX$ with eigenvalue $1/p^i$.

We obtain the recursive definition of $\alpha^{(i)}$ by induction on
$i$. The case $i=0$ follows immediately from~\eqref{oldformula} by
letting $e$ go to infinity. Suppose now that $i>0$ and that the result holds for smaller values of
$i$. From~\eqref{oldformula} and~\eqref{phiongamma} we then get
\begin{align*}
p^{ie}\Phi^e(\alpha-(\alpha^{(0)}+\cdots +\alpha^{(i-1)})) &=
p^{ie}\Phi^e(\alpha^{(i)}+\cdots +\alpha^{(u)}) \\
 &= \alpha^{(i)} +\frac{1}{p^e}\alpha^{(i+1)}+\cdots
 +\frac{1}{p^{e(u-i)}}\alpha^{(u)} ,
\end{align*}
and letting $e$ go to infinity, we obtain the desired formula.

From~\eqref{oldformula} we know that $\alpha^{(0)},\dots ,\alpha^{(u)}$ solve
the following system of equations with rational coefficients.
\begin{equation*}
\begin{matrix}
\alpha^{(0)} &+& \alpha^{(1)} &+&\cdots &+& \alpha^{(u)}
&= & \alpha \\
\alpha^{(0)} &+& \displaystyle\frac{1}{p}\alpha^{(1)} &+& \cdots &+&
\displaystyle\frac{1}{p^{u}}\alpha^{(u)} &= & \Phi_\XXX(\alpha) \\
 \vdots &&  \vdots && \ddots &&  \vdots  &&\vdots   \\
\alpha^{(0)} &+& \displaystyle\frac{1}{p^u}\alpha^{(1)} &+& \cdots &+&
\displaystyle\frac{1}{p^{u^2}}\alpha^{(u)} &= & \Phi_\XXX^u(\alpha)
\end{matrix}
\end{equation*}
Formula~\eqref{generalformula} now follows. (The matrix is the
Vandermonde matrix of the elements $1, 1/p,\dots ,1/p^u$ with determinant $\prod_{0\leq i<j\leq
  u}(1/p^j-1/p^i)\neq 0$.)
\end{proof}

\begin{rema}\label{decompositions}
It is easy to see that, for $\alpha\in\Grot(\XXX)$ and
$\beta\in\Grot(\comp\XXX)$,
\[
(\alpha\otimes\beta)^{(t)}=\sum_{i+j=t}\!\!\!\! \alpha^{(i)}\otimes\beta^{(j)}
\]
in $\Grot(\mmm)$.
In particular, $(\alpha\otimes\beta)^{(0)}=\alpha^{(0)}\otimes\beta^{(0)}$.
Suppose now that $\XXX\subseteq\YYY\subseteq\spec R$ and let
$s=\codim\XXX-\codim\YYY$. Since $\overline{\Phi_\XXX(\alpha)} =
p^{-s}\Phi_\YYY(\overline\alpha)$  in
$\Grot(\YYY)$, it follows from Theorem~\ref{maintheo} that,  in
$\Grot(\YYY)$,
$\overline{\alpha^{(i)}}= \overline\alpha^{(i-s)}$ for $i\geq s$
and $\overline{\alpha^{(i)}}=0$ for $i<s$. 
\end{rema}

\begin{rema}\label{duttacalculations}
The Dutta multiplicity of an element $\alpha\in\Cat(\XXX)$ and complexes in 
$\cat(\comp{\XXX})$ is given by applying the function
\[
\chi_\infty(\alpha,-) =
\lim_{e\to\infty}\chi(\Phi^e_\XXX(\alpha), -) =
\chi(\lim_{e\to\infty}\Phi^e_\XXX(\alpha), -) = \chi(\alpha^{(0)}, -).
\]
Thus, the Dutta multiplicity is a rational number and
we need not find a limit to compute it. 
In fact, translating Theorem~\ref{maintheo} back to the setup
with complexes $X\in\Cat(\XXX)$ and $Y\in\cat(\YYY)$, where
$\XXX=\supp X$, $\YYY=\supp Y$ and $\YYY\subseteq\comp\XXX$, we obtain the general formula
\[
\chi_\infty(X,Y) =
\begin{pmatrix} 1 & 0 & \cdots &  0 \end{pmatrix}
 \begin{pmatrix} 1 & 1  & \cdots & 1 \\ 
 p^t & p^{t-1} & \cdots & p^{t-u} \\
 \vdots & \vdots & \ddots & \vdots \\
 p^{ut} & p^{u(t-1)} & \cdots &
 p^{u(t-u)}\end{pmatrix}^{\!\!\!-1} \!\!\!\begin{pmatrix} \chi(X,Y) \\ \chi(\frob_R(X),Y) \\ \vdots \\
   \chi(\frob_R^u(X),Y) \end{pmatrix} ,
\]
where $t=\codim X$ and $u\geq\vandim [X]$ for $[X]\in\Grot(\XXX)$.
The fact that Dutta multiplicity satisfies vanishing follows
immediately from Proposition~\ref{conditionsdimu}, which extends
Proposition~\ref{conditions} by adding even more conditions that
describe what it means to have a certain vanishing dimension.
 \end{rema}

\begin{prop}\label{conditionsdimu}
Suppose that $\XXX\subseteq\spec R$, let
$\alpha\in\Grot(\XXX)$ and let $u$ be a non-negative integer. The following are equivalent.
\begin{enumerate}
\item \label{C1} $\alpha$ satisfies vanishing.
\item \label{C3} $\alpha=\alpha^{(0)}$.
\item \label{C4} $\alpha=\Phi_\XXX(\alpha)$.
\item \label{C5} $\alpha=\Phi^e_\XXX(\alpha)$ for some $e\in\NN$.
\item \label{C6} $\alpha=\lim_{e\to\infty}\Phi^e_\XXX(\alpha)$.
\end{enumerate}
Further, the following are equivalent.
\begin{enumerate}\setcounter{enumi}{5}
\item \label{D2} $\alpha=\alpha^{(0)}+\cdots +\alpha^{(u)}$.
\item \label{D3} $(p^u\Phi_\XXX -\id)\circ \cdots \circ (p\Phi_\XXX-\id)\circ
  (\Phi_\XXX -\id)(\alpha)=0$.
\item \label{D6} $\vandim\alpha \leq u$.
\end{enumerate}
\end{prop}
\begin{proof}
\eqref{C1} is equivalent to~\eqref{C4} by Proposition~\ref{vanishing};
\eqref{C4} is equivalent to~\eqref{C3} and~\eqref{C6} by
Theorem~\ref{maintheo}; \eqref{C4} implies~\eqref{C5}
implies~\eqref{C6}, so these must all be equivalent;
the proof of Theorem~\ref{maintheo} shows how~\eqref{D6} implies~\eqref{D3} which again
implies~\eqref{D2}; and combining Remark~\ref{decompositions} with
Proposition~\ref{conditions} shows that~\eqref{D2} implies~\eqref{D6}. 
\end{proof}

Having vanishing dimension exactly equal to $u>0$ of course means that
conditions~\eqref{D2}--\eqref{D6} are satisfied and that the same conditions fail
to hold if $u$ is replaced by $u-1$. In particular, if
$\vandim\alpha=u$, then 
$\alpha^{(u)}\neq 0$ and there exists a $\beta\in\grot(\comp{\XXX})$
with $\dim\beta=\codim\XXX-u$ such that
$\alpha\otimes\beta=\alpha^{(u)}\otimes\beta \neq 0$. Consequently,
if the term $\alpha^{(i)}$ is non-zero, then it has vanishing dimension $i$ and can be
regarded as ``the component of $\alpha$ that allows a 
counterexample to vanishing where the difference between co-dimension and
dimension is equal to $i$''.

\section{Numerical vanishing}
\begin{assu}
Throughout this section, we continue to assume that $R$ is complete of
prime characteristic $p>0$, and that $k$ is a perfect field.
\end{assu}

\begin{defi}
Suppose that $\XXX\subseteq\spec R$ and let $\alpha\in\Grot(\XXX)$. We
say that $\alpha$ satisfies \emph{numerical vanishing} if
$\overline\alpha=\overline{\alpha^{(0)}}$ in $\grot(\XXX)$.
\end{defi}

\begin{prop}\label{elementvanishings} Suppose that $\XXX\subseteq\spec
  R$ and let $\alpha\in\Grot(\XXX)$. For the following conditions, each
  condition implies the next.
\begin{enumerate}
\item $\alpha$ satisfies vanishing.
\item $\alpha$ satisfies numerical vanishing
\item $\alpha$ satisfies weak vanishing
\end{enumerate}
\end{prop}
\begin{proof}
It is clear from Proposition~\ref{conditionsdimu} that vanishing implies numerical vanishing. Suppose that
$\alpha$ satisfies numerical vanishing and let $\beta\in\Grot(\comp{\XXX})$ be such that $\dim\beta<\codim\XXX$.
Then
\[
\overline{\alpha\otimes\beta}=\overline\alpha\otimes\beta=\overline{\alpha^{(0)}}\otimes\beta
= \alpha^{(0)}\otimes\overline\beta = 0,
\]
since $\alpha^{(0)}$ satisfies vanishing, and we conclude that
$\alpha$ satisfies weak vanishing.
\end{proof}

As Remark~\ref{lengthrema2} will show, the implications in
Proposition~\ref{elementvanishings} are generally strict.

\begin{rema}\label{lengthrema}
If $X$ is a complex in $\Cat(\mmm)$, then, because of
Proposition~\ref{grothendieckobs}\eqref{iso}, the element $[X]\in\Grot(\mmm)$ satisfies numerical
vanishing if and only if
\begin{equation}\label{lengthcondition}
\lim_{e\to\infty} \frac{1}{p^{e\dim R}} \chi(\frob_R^e(X)) = \chi(X).
\end{equation}
As we shall see in
Proposition~\ref{sufficientcondition} below,
for~\eqref{lengthcondition} to hold, it suffices (but need not be
necessary) to verify that the equation
\[
\chi(\frob_R^e(X)) = p^{e\dim R}\chi(X)
\]
holds for $\vandim[X]$ distinct values of $e>0$.
\end{rema}

\begin{prop} \label{sufficientcondition}
Suppose that $\XXX\subseteq\spec R$ and let $\alpha\in\Grot(\XXX)$. A
sufficient condition for $\alpha$ to satisfy numerical vanishing is
that $\overline\alpha=\overline{\Phi_\XXX^e(\alpha)}$ in $\grot(\XXX)$  for
$\vandim\alpha$ distinct values of $e>0$.
\end{prop}
\begin{proof}
Let $u=\vandim\alpha$. According to Theorem~\ref{maintheo}, the difference
$\overline{\Phi_\XXX^e(\alpha)}-\overline\alpha$ in $\grot(\XXX)$ is
obtained by letting  $x=1/p^e$ in the polynomial  
\[
(\overline{\alpha^{(0)}} - \overline\alpha) +
x\overline{\alpha^{(1)}} +\cdots + x^u\overline{\alpha^{(u)}} .
\]
The polynomial always has the root $x=1$. If there are $u$ additional
roots, it must be the zero-polynomial, so that
$\overline\alpha=\overline{\alpha^{(0)}}$.
\end{proof}

\begin{defi}
We say that $R$ satisfies \emph{vanishing} (or \emph{numerical
 vanishing} or \emph{weak vanishing}, respectively) if all elements
of $\Grot(\XXX)$ satisfy vanishing (or numerical
vanishing or weak vanishing, respectively) for all $\XXX\subseteq\spec
R$.
\end{defi}

\begin{prop} \label{numericalvanishing}
The following are equivalent.
\begin{enumerate}
\item \label{numericalvanishing1} $R$ satisfies numerical vanishing.
\item \label{numericalvanishing2} $\overline\alpha=\overline{\Phi_\XXX(\alpha)}$ in $\grot(\XXX)$ for all
  $\XXX\subseteq\spec R$ and $\alpha\in\Grot(\XXX)$. 
\item \label{numericalvanishing3} $\overline\alpha=\overline{\Phi_\mmm(\alpha)}$ in $\grot(\mmm)$ for all
  $\alpha\in\Grot(\mmm)$.
\item \label{numericalvanishing4} $\overline\alpha=\overline{\alpha^{(0)}}$ in $\grot(\XXX)$ for all
  $\XXX\subseteq\spec R$ and $\alpha\in\Grot(\XXX)$.
\item \label{numericalvanishing5} $\overline\alpha=\overline{\alpha^{(0)}}$ in $\grot(\mmm)$ for all
  $\alpha\in\Grot(\mmm)$.
\end{enumerate}
\end{prop}
\begin{proof}
By definition, \eqref{numericalvanishing1} is equivalent to
\eqref{numericalvanishing4}. It is clear that
\eqref{numericalvanishing2} implies \eqref{numericalvanishing3} and
that \eqref{numericalvanishing4} implies
\eqref{numericalvanishing5}. It is 
also clear that \eqref{numericalvanishing2} implies
\eqref{numericalvanishing4} and that \eqref{numericalvanishing3}
implies \eqref{numericalvanishing5}. Thus, 
it only remains to prove that \eqref{numericalvanishing5} implies
\eqref{numericalvanishing2}. So assume \eqref{numericalvanishing5} and
let $\XXX\subseteq\spec R$ and $\alpha\in\Grot(\XXX)$. Then, for all $\beta\in\Cat(\comp\XXX)$,
\[
\overline{\Phi_\XXX(\alpha)}\otimes\beta=\overline{\Phi_\XXX(\alpha)\otimes\beta}=
\overline{(\Phi_\XXX(\alpha)\otimes\beta)^{(0)}}=
\overline{\Phi_\XXX(\alpha)^{(0)}\otimes\beta^{(0)}}=
\overline{\alpha^{(0)}\otimes\beta^{(0)}},
\]
where we have applied Remark~\ref{decompositions} and the fact that $\Phi_\XXX(\alpha)^{(0)}=\alpha^{(0)}$. Similarly,
\[
\overline{\alpha}\otimes\beta=\overline{\alpha\otimes\beta}=
\overline{(\alpha\otimes\beta)^{(0)}}=
\overline{\alpha^{(0)}\otimes\beta^{(0)}}.
\]
Thus, $\overline\alpha=\overline{\Phi_\XXX(\alpha)}$.
\end{proof}

\begin{rema}\label{lengthrema2}
Comparing Remark~\ref{lengthrema} with
Proposition~\ref{numericalvanishing}, we see that
a necessary and sufficient condition
for $R$ to satisfy numerical vanishing is that
\begin{equation}\label{lengthcondition2}
\chi(\frob_R(X))=p^{\dim R}\chi(X)
\end{equation}
for all complexes $X\in\Cat(\mmm)$, and by
Proposition~\ref{elementvanishings}, this condition implies that
$R$ satisfies weak vanishing.

Dutta~\cite{duttamultiplicity} has proven that
condition~\eqref{lengthcondition2} holds when $R$ is Gorenstein of dimension (at most) $3$ or a complete
intersection (of any dimension). 
The rings in the counterexamples by Dutta, Hochster and
McLaughlin~\cite{dhm} and Miller and Singh~\cite{millersingh} are complete
intersections (which can be assumed to be complete of
characteristic $p$ and with perfect residue fields), and hence they satisfy
numerical vanishing without satisfying vanishing.

Any ring of dimension at most $4$ will satisfy weak vanishing; this
follows from the result by Foxby~\cite{foxbyM}. Roberts~\cite{robertsI} has
shown the existence of a Cohen--Macaulay ring of dimension $3$ (which
can also be assumed to be complete of characteristic $p$ and with perfect residue
field) such that condition~\eqref{lengthcondition2} does not hold. Thus, this ring
satisfies weak vanishing without satisfying numerical vanishing.
\end{rema}

\section*{Acknowledgments}
I am grateful to Anders Frankild for many useful suggestions that
helped improve this paper significantly. I would also like to thank
the anonymous referee for some useful suggestions that 
helped reduce the complexity of this paper. 
Finally, I thank Hans-Bj{\o}rn Foxby and Anurag Singh for reading and commenting the
paper, and Marc Levine for explaining his work to me. 


\providecommand{\bysame}{\leavevmode\hbox to3em{\hrulefill}\thinspace}
\providecommand{\MR}{\relax\ifhmode\unskip\space\fi MR }
\providecommand{\MRhref}[2]{%
  \href{http://www.ams.org/mathscinet-getitem?mr=#1}{#2}
}
\providecommand{\href}[2]{#2}

\end{document}